\documentclass{article}

\usepackage{amsmath}
\usepackage{amsthm}
\usepackage{amsfonts}
\usepackage{amssymb}
\usepackage{graphicx}

\newtheorem{thm}{Theorem}[section]
\newtheorem{cor}[thm]{Corollary}
\newtheorem{lem}[thm]{Lemma}

\theoremstyle{definition}

\theoremstyle{remark}

\begin{document}

\title{On the Multiple Covering Densities of Triangles}

\author{Kirati Sriamorn\footnote{pic\_kirat@hotmail.com}, Akanat Wetayawanich\footnote{miny\_math191@hotmail.com}}

\maketitle

\begin{abstract}
  Given a convex disk $K$ and a positive integer $k$, let $\vartheta_T^k(K)$ and $\vartheta_L^k(K)$ denote the $k$-fold translative covering density and the $k$-fold lattice covering density of $K$, respectively. Let $T$ be a triangle. In a very recent paper, K. Sriamorn \cite{sriamorn} proved that $\vartheta_L^k(T)=\frac{2k+1}{2}$. In this paper, we will show that $\vartheta_T^k(T)=\vartheta_L^k(T)$.
\end{abstract}



\section{Introduction}
Let $D$ be a connected subset of $\mathbb{R}^2$.
A family of bounded sets $\mathcal{F}=\{S_1,S_2,\ldots\}$ is said to be a \emph{k-fold packing} of $D$ if $\bigcup S_i\subset D$ and each point of $D$ belongs to the interiors of at most $k$ sets of the family. In particular, when all $S_i$ are translates of a fixed measurable bounded set $S$ the corresponding family is called a \emph{k-fold translative packing} of $D$ with $S$. When the translation vectors form a lattice the corresponding family is called a \emph{k-fold lattice packing} of $D$ with $S$. Let $I=[0,1)$, and let $M(S,k,l)$ be the maximum number of the bounded sets in a $k$-fold translative packing of $lI^2$ with $S$. Then, we define
$$\delta_T^k(S)=\limsup_{l\rightarrow\infty} \frac{M(S,k,l)|S|}{|lI^2|}.$$
Similarly, we can define $\delta_L^k(S)$ for the $k$-fold lattice packings.

A family of bounded sets $\mathcal{F}=\{S_1,S_2,\ldots\}$ is said to be a \emph{k-fold covering} of $D$ if each point of $D$ belongs to at least $k$ sets of the family. In particular, when all $S_i$ are translates of a fixed measurable bounded set $S$ the corresponding family is called a \emph{k-fold translative covering} of $D$ with $S$. When the translative vectors form a lattice the corresponding family is called a \emph{k-fold lattice covering} of $D$ with $S$. Let $m(S,k,l)$ be the minimum number of the translates in a $k$-fold translative covering of $lI^2$ with $S$. Then, we define
$$\vartheta_T^k(S)=\liminf_{l\rightarrow\infty} \frac{m(S,k,l)|S|}{|lI^2|}.$$
Similarly, we can define $\vartheta_L^k(S)$ for the $k$-fold lattice coverings.

A family $\mathcal{F}=\{S_1,S_2,\ldots\}$ of bounded sets which is both a $k$-fold packing and a $k$-fold covering of $D$ is called a \emph{k-fold tiling} of $D$. In particular, if each point of $D$ belongs to exactly $j$ sets of the family, then we call $\mathcal{F}$ an \emph{exact j-fold tiling} of $D$.

For the case of $1$-fold coverings, Bambah and Rogers \cite{brass_Rogers} conjectured that $\vartheta^1_{T}(K)=\vartheta^1_{L}(K)$ holds for every convex disk $K$.
It is known that this conjecture is true for every centrally symmetric convex disk. Very little is known about the non-symmetric case. In 2010, J. Januszewski \cite{januszewski} showed that for every triangle $T$, $\vartheta^1_T(T)=\vartheta^1_L(T)=\frac{3}{2}$.
In a recent paper, K. Sriamorn and F. Xue  \cite{xuefei} proved that the conjecture is true for a class of convex disks (quarter-convex disks), which includes all triangles and convex quadrilaterals.

In a very recent paper, K. Sriamorn \cite{sriamorn} studied the $k$-fold lattice packings and coverings with  triangles $T$. He proved that
$$\delta_L^k(T)=\frac{2k^2}{2k+1},$$
and
$$\vartheta_L^k(T)=\frac{2k+1}{2}.$$

In this paper, we determine the $k$-fold translative covering densities of triangles. Januszewski's result inspires us to guess the following statement :

\begin{thm} \label{main_thm}
For every triangle $T$, we have $\vartheta_T^k(T)=\vartheta_L^k(T)=\frac{2k+1}{2}.$
\end{thm}

In order to prove the result, we apply the analogous proving approach used in the paper \cite{xuefei}. Firstly, we modify Januszewski's cutting idea and get some related properties. Secondly, by using our cutting method, we cut the triangles of an arbitrary translative covering into stair polygons. Finally, we show that the corresponding stair polygons form an exact $k$-fold tiling of the plane which has special properties that yield the main theorem.

\section{Normal $k$-Fold Translative Covering}
Let $D$ be a connected subset of $\mathbb{R}^2$ and $\mathcal{K}=\{K_1,K_2,\ldots\}$ a family of convex disks. Suppose that $\mathcal{K}$ is a $k$-fold covering of $D$. We say that $\mathcal{K}$ is \emph{normal} provided $K_i\neq K_j$ for all $i\neq j$. When $\mathcal{K}$ is normal and $K_i$ are translates of a fixed convex disk $K$, the corresponding family is called a \emph{normal k-fold translative covering} of $D$ with $K$. Let $\widetilde{m}(K,k,l)$ be the minimum number of the convex disks in a normal $k$-fold translative covering of $lI^2$ with $K$. Then, we define
$$\widetilde{\vartheta}_T^{k}(K)=\liminf_{l\rightarrow\infty}\frac{\widetilde{m}(K,k,l)|K|}{|lI^2|}.$$

\begin{thm}\label{normal_covering}
For every convex disk $K$, we have
$$\widetilde{\vartheta}_T^{k}(K)=\vartheta_T^k(K).$$
\end{thm}
\begin{proof}
Obviously, we have that $\widetilde{\vartheta}_T^{k}(K)\geq\vartheta_T^k(K)$. Let $\{K_1,\ldots K_m\}$ be a $k$-fold translative covering of $lI^2$ with $K$. For any $K_i$, one can see that for every $\varepsilon>0$, there exist infinitely many points $(x,y)$ in the plane such that $K_i\subset (1+\varepsilon)K_i+(x,y)$. Hence, for every $\varepsilon>0$, there exist $m$ points $(x_1,y_1),\ldots,(x_m,y_m)$ in the plane such that $\{(1+\varepsilon)K_1+(x_1,y_1),\ldots,(1+\varepsilon)K_m+(x_m,y_m)\}$ is a normal $k$-fold translative covering of $lI^2$ with $(1+\varepsilon)K$. Therefore, $m\geq \widetilde{m}((1+\varepsilon)K,k,l)$. This implies that $m(K,k,l)\geq \widetilde{m}((1+\varepsilon)K,k,l)$, and hence
\begin{align*}
\vartheta_T^{k}(K)&=\liminf_{l\rightarrow\infty}\frac{m(K,k,l)|K|}{|lI^2|}\\
&\geq\liminf_{l\rightarrow\infty}\frac{\widetilde{m}((1+\varepsilon)K,k,l)|K|}{|lI^2|}\\
&=\frac{1}{(1+\varepsilon)^2}\widetilde{\vartheta}_T^{k}((1+\varepsilon)K)\\
&=\frac{1}{(1+\varepsilon)^2}\widetilde{\vartheta}_T^{k}(K).
\end{align*}
By letting $\varepsilon$ tend to zero, one gets the desired result.
\end{proof}

\section{Definitions of Cutting and Stair Polygons}\label{def_of_cut}
Denote by $T$ the triangle with vertices $(0,0)$, $(1,0)$ and $(0,1)$.
If $T'=T+(x,y)$ where $(x,y)\in\mathbb{R}^2$, then we denote by $v_{T'}$ the vertex $(x,y)$ of $T'$.

For $(x_1,y_1),(x_2,y_2)\in\mathbb{R}^2$, we define the relation $\prec$ by $(x_1,y_1)\prec (x_2,y_2)$ if and only if either
$$x_1+y_1<x_2+y_2$$
or
$$x_1+y_1=x_2+y_2 \text{~and~} x_1<x_2.$$
One can easily show that $\prec$ is a strict partial ordering over $\mathbb{R}^2$.

Suppose that $T_1$ and $T_2$ are two distinct translates of $T$ and $T_1\cap T_2\neq\emptyset$. We say that $T_1$ cuts $T_2$ provided $v_{T_2}\prec v_{T_1}$ (Fig. \ref{typeofcut}).

\begin{figure}[!ht]
  \centering
    \includegraphics[scale=.70]{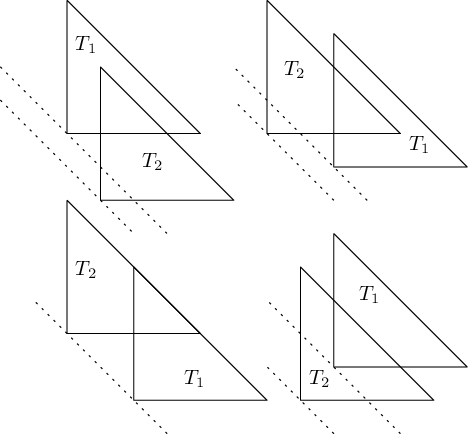}
   \caption{Types of cutting ($T_1$ cuts $T_2$)}\label{typeofcut}
\end{figure}

By the definition, one gets the following statements :

\begin{itemize}
\item Suppose that $T_1$ and $T_2$ are distinct translates of $T$. If $T_1\cap T_2\neq\emptyset$ then either $T_1$ cuts $T_2$ or $T_2$ cuts $T_1$.
\item Suppose that $T_1,~T_2$ and $T_3$ are three distinct translates of $T$ and $T_1\cap T_2\cap T_3\neq\emptyset$. If $T_1$ cuts $T_2$ and $T_2$ cuts $T_3$, then $T_1$ cuts $T_3$.
\item Suppose that $T_1,\ldots,T_n$ are $n$ distinct translates of $T$ and $T_1\cap\cdots\cap T_n\neq\emptyset$. Then, there exists $i\in\{1,\ldots,n\}$ such that $T_j$ cuts $T_i$ for all $j\neq i$.
\end{itemize}

For a non-negative integer $r$, we call a planar set $S$ a \emph{half-open r-stair polygon} (Fig. \ref{stairpolygon}) if there are $x_0<x_1<\cdots< x_{r+1}$ and $y_0>y_1>\cdots > y_r>y_{r+1}$ such that
$$S=\bigcup_{i=0}^r[x_i,x_{i+1})\times[y_{r+1},y_i).$$

\begin{figure}[!ht]
  \centering
    \includegraphics[scale=.70]{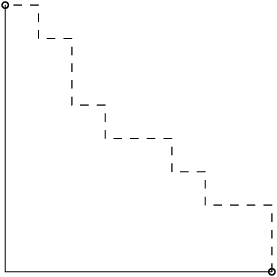}
   \caption{A half-open $5$-stair polygon}\label{stairpolygon}
\end{figure}

Let $A(r)$ denote the maximum area of a half-open $r$-stair polygon contained in $T$. Clearly, $A$ is an increasing function. By elementary calculations, one can obtain
\begin{equation}
A(r)=\frac{r+1}{2(r+2)},
\end{equation}
where $r=0,1,2,\ldots$. Let $B$ be the function on $[0,+\infty)$ defined by
\begin{equation}
B(x)=\frac{x+1}{2(x+2)},
\end{equation}
It is obvious that $B$ is an increasing concave function and $B(r)=A(r)$ , for all $r=0,1,2,\ldots$. For convenience, we also denote the function $B$ by $A$. In \cite{sriamorn}, K. Sriamorn showed that
\begin{equation}\label{lattice_density}
\vartheta_L^k(T)=\frac{k|T|}{A(2k-1)}=\frac{2k+1}{2}.
\end{equation}
\section{Cutting Triangles into Stair Polygons}
In this section, we assume that $\mathcal{T}=\{T_1,T_2,\ldots\}$ is a normal $k$-fold translative covering of $\mathbb{R}^2$ with $T$. Moreover, we assume that $\mathcal{T}$ is \emph{locally finite}, i.e., for every $l>0$,
any translate of $lI^2$ is intersected by a finite number of   members of $\mathcal{T}$ only.

Denote by $\mathcal{C}_i$ the collection of triangles $T_j$ that cut $T_i$. Since $\mathcal{T}$ is locally finite, we know that $\mathcal{C}_i$ is finite . Let
\begin{equation*}
U_i=\bigcup_{\substack{T_{i_1},\ldots,T_{i_k}\in\mathcal{C}_i\\ i_1,\ldots,i_k \text{~are distinct}}}(T_{i_1}\cap\cdots\cap T_{i_k}),
\end{equation*}
and
$$S_i=T_i\setminus U_i.$$

For any right triangle $T'$, we denote by $H(T')$ for the hypotenuse of $T'$. One can show that for every $T_{i_1},\ldots,T_{i_k}\in\mathcal{C}_i$, $T_i\cap T_{i_1}\cap\ldots\cap T_{i_k}$ is a right triangle ( if not empty ) similar to $T$ and  $H(T_i\cap T_{i_1}\cap\ldots\cap T_{i_k})\subset H(T_i)$.

\begin{lem}\label{S_i_stair_polygon}
$S_i$ is a half-open stair polygon.
\end{lem}
\begin{proof} It is easy to see that for every $T'\in\mathcal{C}_i$, $v_{T_i}\notin T'$. Hence, it is obvious that $v_{T_i}\in S_i$. Therefore $S_i\neq\emptyset$. It suffices to show that $H(T_i)\subset U_i$. Suppose that $(x,y)\in H(T_i)$. Let
$$\mathcal{F}_1=\{T'\in\mathcal{T}: (x,y)\in T'\text{~and~}(x+\varepsilon,y+\varepsilon)\in T' \text{~for some positive~}\varepsilon<1\}$$
and
$$\mathcal{F}_2=\{T'\in\mathcal{T}: (x,y)\notin T'\text{~and~}(x+\varepsilon,y+\varepsilon)\in T' \text{~for some positive~}\varepsilon<1\}$$
Since $\mathcal{T}$ is locally finite, we know that $\mathcal{F}_2$ is a finite collection. Therefore there exists a positive $\varepsilon_0<1$ such that $(x+\varepsilon_0,y+\varepsilon_0)\notin T'$ for all $T'\in\mathcal{F}_2$. But $\mathcal{T}$ is a $k$-fold covering of $\mathbb{R}^2$, so there must exist at least $k$ distinct translates of $T$ that contain $(x+\varepsilon_0,y+\varepsilon_0)$. Hence $card\{\mathcal{F}_1\}\geq k$. One can observe that for all $T'\in\mathcal{F}_1$, $T'$ cuts $T_i$. It follows that $(x,y)\in U_i$.
\end{proof}

\begin{lem}\label{not_exceed_k_1}
Suppose that $i_1,\ldots,i_{k+1}$ are $k+1$ distinct positive integers. Then
$$\bigcap_{j=1}^{k+1}S_{i_j}=\emptyset.$$
\end{lem}
\begin{proof}
If $T_{i_1}\cap\cdots\cap T_{i_{k+1}}=\emptyset$, then it is obvious that $S_{i_1}\cap\cdots\cap S_{i_{k+1}}= \emptyset$. Assume that $T_{i_1}\cap\cdots\cap T_{i_{k+1}}\neq \emptyset$. Without loss of generality,  we may assume that $T_{i_j}$ cuts $T_{i_{k+1}}$ for all $j=1,\ldots,k$. Therefore $T_{i_1}\cap\cdots\cap T_{i_{k}}\subset U_{i_{k+1}}$. Hence
$$\bigcap_{j=1}^{k+1}S_{i_j}=(T_{i_1}\cap\cdots\cap T_{i_{k+1}})\setminus(U_{i_1}\cup\cdots\cup U_{i_{k+1}})= \emptyset.$$
\end{proof}

\begin{lem}\label{at_least_k_S_contain}
For every point $(x,y)$ in $\mathbb{R}^2$, there exist at least $k$ distinct sets $S_{i_1},\ldots,S_{i_{k}}$ such that $(x,y)\in S_{i_j}$, for $j=1,\ldots,k$.
\end{lem}
\begin{proof}
Let $\mathcal{F}=\{T'\in\mathcal{T} : (x,y)\in T'\}$. Since $\mathcal{T}$ is locally finite, we know that $\mathcal{F}$ is a finite collection. We may assume that $\mathcal{F}=\{T_{i_1},T_{i_2},\ldots,T_{i_l}\}$ where $v_{T_{i_l}}\prec\cdots\prec v_{T_{i_2}}\prec v_{T_{i_1}}$. Since $\mathcal{T}$ is a $k$-fold covering of $\mathbb{R}^2$, we have $l\geq k$. One can easily show that $(x,y)\notin U_{i_j}$ for $j=1,\ldots,k$. Hence $(x,y)\in S_{i_j}$ for $j=1,\ldots,k$.
\end{proof}

\begin{cor}\label{S_i_exact_k_fold}
$\{S_i\}$ is an exact $k$-fold tiling of $\mathbb{R}^2$.
\end{cor}

\begin{lem} \label{to_prove_no_R_cut_by_other}
Let $\overline{S_i}$ be the closure of $S_i$ and $R_i=\overline{S_i}\setminus S_i$. If $T_i$ cuts $T_j$, then $R_i\cap S_j=\emptyset$ (Fig. \ref{RS}).
\end{lem}
\begin{proof}
Assume that there is some point $(x,y)\in R_i\cap S_j$. We have that $(x,y)\in U_i$, and hence there exist $T_{i_1},\ldots,T_{i_k}\in \mathcal{C}_i$ such that $(x,y)\in T_{i_1}\cap\cdots\cap T_{i_k}$. Since $T_i$ cuts $T_j$ and $(x,y)\in T_j$, we obtain that $T_{i_1},\ldots,T_{i_k}\in \mathcal{C}_j$. Therefore $(x,y)\in U_j$, which is a contradiction.
\end{proof}

\begin{figure}[!ht]
  \centering
    \includegraphics[scale=.70]{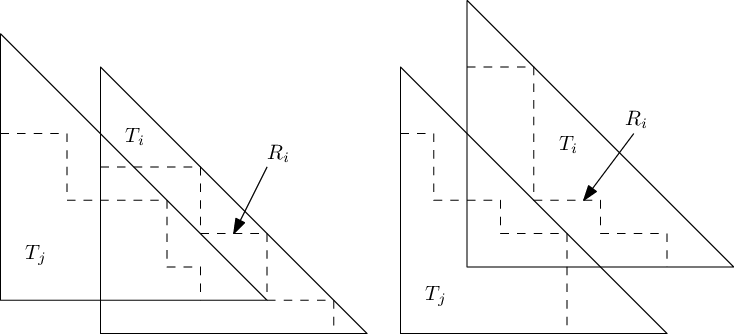}
   \caption{$R_i\cap S_j=\emptyset$}\label{RS}
\end{figure}

\begin{cor}
For every $i,j$, we have $R_i\cap S_j=\emptyset$ or $R_j\cap S_i=\emptyset$.
\end{cor}

If $\mathcal{T}=\{T_1,\ldots,T_N\}$ is a $k$-fold translative covering of $lI^2$ with $T$, then we define
$$S_i=lI^2\cap(T_i\setminus U_i).$$
We note that $S_i$ may be empty for some $i$. But when $S_i\neq\emptyset$, the statement of  Lemma \ref{S_i_stair_polygon} is still true, i.e., $S_i$ is a half-open stair polygon. Furthermore, it is not hard to see that all of the remaining lemmas and corollaries in this section are also true for the modified $S_i$ if $\mathbb{R}^2$ is changed to $lI^2$.

\section{$k$-Fold Tiling of $lI^2$ with Stair Polygons}
In this section, we suppose that $\{S_1,\ldots,S_N\}$ is an exact $k$-fold tiling of $lI^2$ where $S_i$ is a half-open stair polygon. Also, we assume that for every $i,j\in\{1,\ldots,N\}$, we have
\begin{equation}\label{not_cut_property}
R_i\cap S_j=\emptyset \text{~or~} R_j\cap S_i=\emptyset
\end{equation}
where $R_i=\overline{S_i}\setminus S_i$.

We may assume, without loss of generality, that $S_i$ is a half-open $r_i$-stair polygon and that
$$S_i=\bigcup_{j=0}^{r_i}[x_j^{(i)},x_{j+1}^{(i)})\times[y_{r_i+1}^{(i)},y_j^{(i)}),$$
where $x_0^{(i)}<x_1^{(i)}<\cdots<x_{r_i+1}^{(i)}$ and $y_0^{(i)}>y_1^{(i)}>\cdots>y_{r_i+1}^{(i)}$ (Fig. \ref{Si}). Denote by $v_{S_i}$ the vertex $(x_0^{(i)},y_{r_i+1}^{(i)})$ of $S_i$. Let
$$Z(S_i)=\{(x_j^{(i)},y_j^{(i)}): j=1,\ldots,r_i\}.$$

\begin{figure}[!ht]
  \centering
    \includegraphics[scale=1]{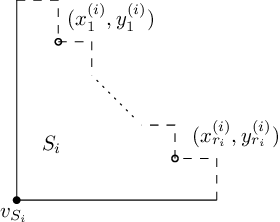}
   \caption{$S_i$}\label{Si}
\end{figure}

We will prove the following result.

\begin{thm}\label{approximate_average_stair_point}
$$\sum_{i=1}^{N}r_i\leq (2k-1)N.$$
\end{thm}

First , we prove the following lemmas.

\begin{lem}\label{x_coordinate}
For every $(x,y)\in Z(S_i)$, there exists a $j\in\{1,\ldots,N\}\setminus\{i\}$ such that $(x,y)\in S_j$ and $x=x_0^{(j)}$ where $x_0^{(j)}$ is the $x$-coordinate of $v_{S_j}$ (Fig. \ref{SiSj}).
\end{lem}
\begin{proof}
Let
$$\mathcal{F}=\{S_j : (x,y)\in S_j, j=1,\ldots,N\},$$
and
$$x_0=\max\{x' : x' \text{~is the~}x\text{-coordinate of~}v_{S_j}, S_j\in\mathcal{F}\}.$$
Obviously $x_0\leq x$. Furthermore, since $\{S_1,\ldots,S_N\}$ is an exact $k$-fold tiling of $lI^2$ and $(x,y)\in Z(S_i)\subset lI^2$, we know that $card\{\mathcal{F}\}=k$ and $S_i\notin \mathcal{F}$. If $x_0<x$, then one can show that
$$S_i\cap \bigcap_{S\in\mathcal{F}}S\neq \emptyset.$$
This is impossible, since $\{S_1,\ldots,S_N\}$ is an exact $k$-fold tiling of $lI^2$. Hence $x_0=x$.
\begin{figure}[!ht]
  \centering
    \includegraphics[scale=.85]{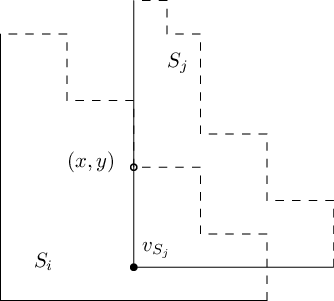}
   \caption{$S_i$ and $S_j$}\label{SiSj}
\end{figure}
\end{proof}
\begin{lem}\label{more_than_r_k_1}
For $i=1,\ldots,N$, let
$$n_i=card\{S_j : v_{S_j}\in Int(S_i)\cup Z(S_i), j=1,\ldots,N\}.$$
Then, we have
$$n_i\geq r_i-k+1.$$
\end{lem}
\begin{proof}
Suppose that $Z(S_i)=\{(x_1^{(i)},y_1^{(i)}),\ldots,(x_{r_i}^{(i)},y_{r_i}^{(i)})\}$. From Lemma \ref{x_coordinate}, we know that for every $j=1,\ldots,r_i$, there exists an $i_j\in\{1,\ldots,N\}\setminus\{i\}$ such that $(x_j^{(i)},y_j^{(i)})\in S_{i_j}$ and $x_j^{(i)}=x_{i_j}$, where $x_{i_j}$ is the $x$-coordinate of $v_{S_{i_j}}$ (see Fig. \ref{Sij}). Let $y_i$ and $y_{i_j}$ be the $y$-coordinates of $v_{S_i}$ and $v_{S_{i_j}}$, respectively. Let
$$\mathcal{F}=\{S_{i_j} :y_{i_j}\leq y_i, j=1,\ldots,r_i\}.$$
From (\ref{not_cut_property}), we know that $R_{i_j}\cap S_i=\emptyset$ for all $S_{i_j}\in\mathcal{F}$. We note that $S_i\notin \mathcal{F}$. Since $\{S_1,\ldots,S_N\}$ is an exact $k$-fold tiling of $lI^2$, one can deduce that $card\{\mathcal{F}\}\leq k-1$. It is not hard to see that for every $S\in\{S_{i_1},S_{i_2},\ldots,S_{i_{r_i}}\}\setminus \mathcal{F}$, we have $v_S\in Int(S_i)\cup Z(S_i)$. Hence
$$n_i\geq card\{Z(S_i)\}-card\{\mathcal{F}\}\geq r_i-k+1.$$
\begin{figure}[!ht]
  \centering
    \includegraphics[scale=.85]{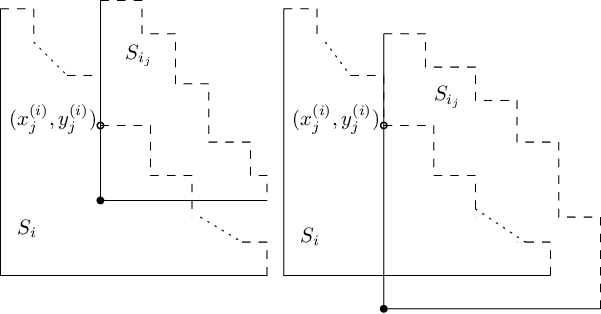}
   \caption{$S_{i_j}$}\label{Sij}
\end{figure}
\end{proof}

\begin{lem}\label{less_than_kN}
$$\sum_{i=1}^{N}n_i\leq kN$$
\end{lem}
\begin{proof}
For $i=1,\ldots,N$, let $\mathcal{F}_i=\{S_j : v_{S_j}\in Int(S_i)\cup Z(S_i), j=1,\ldots,N\}$ and $\mathcal{F}^*_i=\{S_j : v_{S_i}\in Int(S_j)\cup Z(S_j), j=1,\ldots,N\}$. Clearly, we have $n_i=card\{\mathcal{F}_i\}$. Let $n_i^*=card\{\mathcal{F}^*_i\}$. It is not hard to show that $\sum_{i=1}^{N}n_i=\sum_{i=1}^{N}n_i^*$. On the other hand,
since $\{S_1,\ldots,S_N\}$ is an exact $k$-fold tiling of $lI^2$, it is obvious that $n_i^*\leq k$. Hence
$$\sum_{i=1}^{N}n_i=\sum_{i=1}^{N}n_i^*\leq kN$$
\end{proof}
Theorem \ref{approximate_average_stair_point} follows immediately from Lemma \ref{more_than_r_k_1} and Lemma \ref{less_than_kN}.

\section{Proof of Main Theorem}
If Theorem \ref{main_thm} were not true, then $\vartheta_T^k(T)<\vartheta_L^k(T)$ would hold.
By (\ref{lattice_density}) and the definition of $\vartheta_T^k$, there would exist a sufficiently large $l>0$ such that
$$\frac{m(T,k,l)|T|}{|lI^2|}<\frac{k|T|}{A(2k-1)},$$
i.e.,
$$|lI^2|>\frac{m(T,k,l)}{k}A(2k-1).$$
Hence, to prove Theorem \ref{main_thm}, it suffices to show that for every $l>0$ if $\mathcal{T}=\{T_1,\ldots,T_N\}$ is a $k$-fold translative covering of $lI^2$ with $T$, then
$$|lI^2|\leq \frac{N}{k}A(2k-1).$$

By Theorem \ref{normal_covering}, we may assume that $\mathcal{T}$ is a normal $k$-fold translative covering of $lI^2$. We define
$$S_i=lI^2\cap(T_i\setminus U_i),$$
where $i=1,\ldots,N$. Without loss of generality, we may assume that $S_i\neq\emptyset$ and $S_j=\emptyset$ when $i=1,\ldots,N'$ and $j=N'+1,\ldots,N$. We know that for all $i=1,\ldots,N'$, $S_i$ is a half-open $r_i$-stair polygon contained in $T_i$. Furthermore,  $\{S_1,\ldots,S_{N'}\}$ is an exact $k$-fold tiling of $lI^2$ and for every $i,j\in\{1,\ldots,N'\}$, $R_i\cap S_j=\emptyset$ or $R_j\cap S_i=\emptyset$, where $R_i=\overline{S_i}\setminus S_i$. By Theorem \ref{approximate_average_stair_point}, we know that
\begin{equation}
\sum_{i=1}^{N'}r_i\leq (2k-1)N'.
\end{equation}
Hence, by the definition and properties of the function $A$, one obtains
\begin{align*}
|lI^2|&=\frac{1}{k}\sum_{i=1}^{N'}|S_i|\\
&\leq \frac{1}{k}\sum_{i=1}^{N'}A(r_i)\\
&\leq \frac{N'}{k}A\left(\frac{1}{N'}\sum_{i=1}^{N'}r_i\right)\\
&\leq \frac{N'}{k}A(2k-1)\\
&\leq \frac{N}{k}A(2k-1)
\end{align*}
This completes the proof.
\section*{Acknowledgment}
This work was supported by 973 programs 2013CB834201 and 2011CB302401.

\end{document}